\documentclass[12pt,reqno]{amsart}
\usepackage{enumerate, latexsym, amsmath, amsfonts, amssymb, amsthm, graphicx, color}
\def\pmod #1{\ ({\rm{mod}}\ #1)}
\def\Z{\Bbb Z}

\def\Q{\Bbb Q}

\def\F{\Bbb F}

\def\r{\right}
\def\bg{\bigg}
\def\({\bg(}
\def\){\bg)}

\def\sgn{{\rm sgn}}

\def\sgn{{\rm sgn}}
\def\Arg{{\rm Arg}}

\def\ve{\varepsilon}

\def\Ack{\medskip\noindent {\bf Acknowledgments}}
\theoremstyle{plain}
\newtheorem{theorem}{Theorem}

\newtheorem{lemma}{Lemma}

\theoremstyle{definition}

\theoremstyle{remark}
\newtheorem{remark}{Remark}

\makeatletter
\@namedef{subjclassname@2010}{%
  \textup{2010} Mathematics Subject Classification}
\makeatother
 \vspace{4mm}

\begin{document}
 \baselineskip=17pt
\hbox{} {}
\medskip
\title[Jacobsthal sums and permutations of biquadratic residues]
{Jacobsthal sums and permutations of biquadratic residues}
\date{}
\author[Hai-Liang Wu and Yue-Feng She] {Hai-Liang Wu and Yue-Feng She}

\thanks{
\newline \indent Supported by the National Natural Science
Foundation of China (Grant No. 11971222).}

\address {(Hai-Liang Wu) School of Science, Nanjing University of Posts and Telecommunications, Nanjing 210023, People's Republic of China}
\email{\tt whl.math@smail.nju.edu.cn}

\address {(Yue-Feng She) Department of Mathematics, Nanjing
University, Nanjing 210093, People's Republic of China}
\email{{\tt she.math@smail.nju.edu.cn}}

\begin{abstract}
Let $p\equiv1\pmod 4$ be a prime. In this paper, with the help of Jacobsthal sums over finite fields, we study some permutation problems involving biquadratic residues modulo $p$.
\end{abstract}
\maketitle
\section{Introduction}
\setcounter{lemma}{0}
\setcounter{theorem}{0}
\setcounter{corollary}{0}
\setcounter{remark}{0}
\setcounter{equation}{0}
\setcounter{conjecture}{0}
\setcounter{proposition}{0}
Let $p$ be an odd prime, and let $k$ be a positive integer with $k\mid p-1$. For every integer $a$ with
$p\nmid a$, we say that $a$ is a $k$-th power residue modulo $p$ if there is an $x\in\Z$ such that
$x^k\equiv a\pmod p$. When $k=2$, the law of quadratic reciprocity is one of the most beautiful results in number theory. Along this line, many mathematicians investigated the higher reciprocity laws.
For example, Gauss formulated and discovered the law of biquadratic reciprocity. Later G. Eisenstein gave
the first complete proofs of cubic and biquadratic reciprocity. In general, as an application of class field theory, we have the following elegant {\bf general reciprocity law for $k$-th power residues}
(cf. \cite[Theorem 8.3, p. 415]{N}).

Let $K$ be a number field containing the group of $k$-th roots of unity, and let $\mathcal{O}$ denote the ring of algebraic integers in $K$. Let $a,b\in K^{\times}$.
Suppose that the principle ideals $a\mathcal{O}$ and $b\mathcal{O}$ are relatively prime and that the ideals $ab\mathcal{O}$ and $k\mathcal{O}$ are prime to each other. Then we have
\begin{equation*}
\(\frac{a}{b}\)_k\(\frac{b}{a}\)_k^{-1}=\prod_{\mathfrak{p}\mid k\infty}\(\frac{a,b}{\mathfrak{p}}\)_k,
\end{equation*}
where $(\frac{}{\mathfrak{p}})_k$ is the $k$-th power residue symbol and
$(\frac{,}{\mathfrak{p}})_k$ is the Hilbert symbol. On the other hand, many mathematicians  focused on the rational reciprocity laws. For each $a\in\Z$ we define the {\bf rational $4$-th power residue symbol} as follows.
$$\(\frac{a}{p}\)_4=\begin{cases}0&\mbox{if}\ p\mid a,
\\1&\mbox{if}\ \text{{\it a} is a 4-th power residue modulo {\it p}} ,
\\-1&\mbox{otherwise}.\end{cases}$$
K. Burde \cite{B} in 1969 obtained the following elegant rational quartic reciprocity law.

Let $p=a^2+b^2$ and $q=A^2+B^2$ be odd primes with $a\equiv A\equiv1\pmod4$.
Suppose that $p$ is a quadratic residue modulo $q$. Then we have
\begin{equation*}
\(\frac{p}{q}\)_4\(\frac{q}{p}\)_4=(-1)^{\frac{q-1}{4}}\(\frac{aB-bA}{q}\)_2.
\end{equation*}
For more details on rational reciprocity laws the readers may consult \cite{L,LE}.

On the other hand, $k$-th power residues modulo $p$ have close relations with the permutations in the finite field $\F_p$.
We first introduce some previous work in this area. In \cite{Z} the well-known Zolotarev's Lemma shows that the Legendre symbol $(\frac{a}{p})$ coincides with the sign of the permutation of $\Z/p\Z$ induced by multiplication by $a$. With the help of this, G. Zolotarev obtained a simple proof of the law of quadratic reciprocity. Later W. Duke and K. Hopkins \cite{DH} extended G. Zolotarev's result to any finite groups. They also generalized the law of quadratic reciprocity to finite groups. In 2019, Z.-W. Sun posed the following permutation problems involving quadratic residues modulo $p$.
Let $c_1,\cdots,c_n$ be a sequence of all the $n=(p-1)/2$ quadratic residues among $1,\cdots,p-1$ in the ascending order. The sequence
\begin{align}\label{sequence A}
\{1^2\}_p,\cdots,\{n^2\}_p.
\end{align}
is a permutation $\sigma_p$ on
\begin{align}\label{sequence B}
c_1,\cdots,c_n,
\end{align}
where $\{a\}_p$ denotes the least nonnegative residue of $a$ modulo $p$.
Z.-W. Sun \cite{S} first investigated this permutation and he obtained that
$$\sgn(\sigma_p)=\begin{cases}1&\mbox{if}\ p\equiv 3\pmod8,\\(-1)^{(h(-p)+1)/2}&\mbox{if}\ p\equiv 7\pmod8,\end{cases}$$
where $h(-p)$ is the class number of $\Q(\sqrt{-p})$ and $\sgn(\sigma_p)$ is the sign of $\sigma_p$.
Inspired by Z.-W. Sun's work, the first author \cite{Wu} studied the following permutation problem.
Let $g\in\Z$ be a primitive root modulo $p$. Then the sequence
\begin{align}\label{sequence C}
\{g^2\}_p,\{g^4\}_p,\cdots,\{g^{p-1}\}_p
\end{align}
is a permutation of the sequence (\ref{sequence A}). The first author determined the sign of this permutation in the case $p\equiv1\pmod4$.

We consider the following problem involving biquadratic residues.
Let $p\equiv1\pmod8$ be a prime. Let $0<a_1<a_2<\cdots<a_{(p-1)/4}<p/2$ be all the quadratic residues modulo $p$ in the interval $(0,p/2)$.
Clearly $\{a_1^2\}_p,\{a_2^2\}_p,\cdots,\{a_{(p-1)/4}^2\}_p$ are exactly all the biquadratic residues modulo $p$ in the interval $(0,p)$.
Let $g\in\mathbb{Z}$ be a primitive root modulo $p$. Consider the following sequences.
\begin{align}\label{sequence D}
\{a_1^2\}_p,\{a_2^2\}_p,\cdots,\{a_{(p-1)/4}^2\}_p,
\end{align}
\begin{align}\label{sequence E}
\{g^4\}_p,\{g^{8}\}_p,\cdots,\{g^{p-1}\}_p.
\end{align}
\begin{align}\label{sequence F}
b_1,\cdots,b_{(p-1)/4},
\end{align}
where $b_1,\cdots,b_{(p-1)/4}$ are all the biquadratic residues modulo $p$ in the interval $(0,p)$ in ascending order.
Clearly (\ref{sequence E}) and (\ref{sequence F}) are permutations of (\ref{sequence D}) and we call these permutations $\tau_{p}(g)$ and $\rho_p$ respectively. The signs of $\tau_p(g)$ and $\rho_p$ are denoted by $\sgn(\tau_p(g))$ and $\sgn(\rho_p)$ respectively. In this paper, we mainly focus on some congruences and some permutation problems induced by $4$-th power residues modulo $p$. Before the
statement of our first result, we introduce the following notations. We let
$$\Omega_1:=\{0<x<p/2: x^{(p-1)/4}\equiv1\pmod p\},$$
and let
$$\Omega_{-1}:=\{0<x<p/2: x^{(p-1)/4}\equiv-1\pmod p\}.$$
With the above notations, we define
$A_p=\prod_{t\in\Omega_1}t$ and $B_p=\prod_{t\in\Omega_{-1}}t.$ We will see in Lemma \ref{lemma Np t} that
$$A_p^2\equiv(-1)^{\frac{p+7}{8}}\pmod p,$$ and
$$B_p^2\equiv (-1)^{\frac{p-1}{8}}\pmod p.$$
From this, when $p\equiv9\pmod {16}$, we let $\beta_p\in\{0,1\}$ satisfying
$$(-1)^{\beta_p}\equiv A_p\pmod p.$$
Let $g\in\Z$ be an arbitrary primitive root modulo $p$. When $p\equiv 1\pmod {16}$, we let $\delta_p(g)\in\{0,1\}$ satisfying
$$(-1)^{\delta_p(g)}\equiv A_p/g^{(p-1)/4}\pmod p.$$
For each prime $p$ with $p\equiv1\pmod4$, it is known that $p$ can be uniquely written as
$$p=a^2+4b^2$$
with $a\equiv-1\pmod 4$ and $b>0$. Clearly $(2b/a)^2\equiv-1\pmod p$. This shows that
$2b/a\pmod p$ is a primitive $4$-th root of unity in $(\Z/p\Z)^{\times}.$ With these notations, when
$p\equiv9\pmod{16}$, we let $\gamma_p\in\{0,1\}$ satisfying
$$(-1)^{\gamma_p}\equiv 2bB_p/a\pmod p.$$
For each primitive root $g$ modulo $p$, when $p\equiv1\pmod{16}$, we let $\mu_p(g)\in\{0,1\}$ satisfying
$$(-1)^{\mu_p(g)}\equiv 2bB_p/(ag^{\frac{p-1}{4}})\pmod p.$$
To state our result, we let
$$\lambda_p=\begin{cases}0&\mbox{if}\ U_p>0,\\1&\mbox{if}\ U_p<0,\end{cases}$$
where $U_p$ denotes the product
$$\prod_{1\le i<j\le(p-1)/4}\sin\frac{2\pi(a_j^2-a_i^2)}{p}.$$
We also let $\#S$ denote the cardinality of a finite set $S$ and use the symbol $\chi_4(\cdot)$ to denote the rational $4$-th power residue symbol.

Now we are in the position to state our first result.
\begin{theorem}\label{theorem A}
Let $p\equiv1\pmod8$ be a prime of the form $a^2+4b^2$ with $a\equiv-1\pmod 4$ and $b>0$. Then we have the following results.

{\rm (i)} If $p\equiv9\pmod{16}$, then $\sgn(\tau_p(g))$ is independent on the choice of $g$. And we have
$$\sgn(\tau_p(g))=\begin{cases}
(-1)^{\lambda_p+\beta_p+\frac{3p+11-2a+8b}{32}}&\mbox{if}\ \chi_4(2)=+1,
\\(-1)^{\lambda_p+\gamma_p+\frac{3p-21-2a+8b}{32}}&\mbox{if}\ \chi_4(2)=-1.\end{cases}$$

{\rm (ii)} If $p\equiv 1\pmod {16}$, then we have
$$\sgn(\tau_p(g))=\begin{cases}
(-1)^{\lambda_p+\delta_p(g)+\frac{3p-5-2a+8b}{32}}&
\mbox{if}\ \chi_4(2)=+1,
\\(-1)^{\lambda_p+\mu_p(g)+\frac{3p+27-2a+8b}{32}}&\mbox{if}\ \chi_4(2)=-1.\end{cases}$$
In this case, if we let
$$\mathcal{P}=\{0<g<p: \text{$g$ is a primitive root modulo $p$}\},$$
then we have
$$\#\{g\in\mathcal{P}:\sgn(\tau_p(g))=1\}=\#\{g\in\mathcal{P}:\sgn(\tau_p(g))=-1\}.$$
\end{theorem}
To state our second result, we let
$$\ve_p=\#\bigg\{(x,y): 0<x,y,x+y<p/2,\(\frac{x}{p}\)_4=\(\frac{y}{p}\)_4=1\bigg\}.$$
\begin{theorem}\label{theorem B}
Let $p\equiv1\pmod8$ be a prime. Then we have
$$\sgn(\rho_p)=(-1)^{\lambda_p+\ve_p}.$$
\end{theorem}
\begin{remark}
It is clear that $\sgn(\rho_p)=1$ if and only if
$$\prod_{1\le i<j\le(p-1)/4}(\{a_j^2\}_p-\{a_i^2\}_p)>0.$$
From this, we immediately obtain
$$\#\big\{(i,j): 1\le i<j\le\frac{p-1}{4},\{a_i^2\}_p>\{a_j^2\}_p\big\}\equiv \lambda_p+\ve_p\pmod2.$$
\end{remark}

We will prove Theorem \ref{theorem A} in Section 2. The proof of Theorem \ref{theorem B} will be given in Section 3.
\maketitle
\section{Proof of Theorem \ref{theorem A}}
\setcounter{lemma}{0}
\setcounter{theorem}{0}
\setcounter{corollary}{0}
\setcounter{remark}{0}
\setcounter{equation}{0}
\setcounter{conjecture}{0}
We first introduce the definition of Jacobsthal sums. Readers may see \cite{BEW} for details.

Let $m\in\Z$ with $p\nmid m$. For all $k\in\Z^{+}$,
the Jacobsthal sums $\phi_k(m)$ and $\psi_k(m)$ are defined by
\begin{align*}
\phi_k(m)&=\sum_{x=1}^{p-1}\(\frac{x}{p}\)\(\frac{x^k+m}{p}\),
\\\psi_k(m)&=\sum_{x=1}^{p-1}\(\frac{x^k+m}{p}\).
\end{align*}
We need the following result concerning Jacobsthal sums.
\begin{lemma}\label{lemma phi 2 m}
Let $p\equiv1\pmod4$ be a prime of the form $a^2+4b^2$ with $a\equiv-1\pmod p$ and $b>0$. For all $m\in\Z$ with $p\nmid m$ we have

{\rm (i)} {\rm (\cite[Theorem 6.2.9]{BEW})} The explicit value of $\phi_2(m)$ is given by
$$\phi_2(m)=\begin{cases}2a&\mbox{if}\ m^{\frac{p-1}{4}}\equiv1\pmod p,\\4b&\mbox{if}\
m^{\frac{p-1}{4}}\equiv2b/a\pmod p,\\-2a&\mbox{if}\ m^{\frac{p-1}{4}}\equiv-1\pmod p,\\
-4b&\mbox{if}\ m^{\frac{p-1}{4}}\equiv-2b/a\pmod p.\end{cases}$$

{\rm (ii)} {\rm (\cite[Theorem 6.1.13]{BEW})} The following equality holds.
$$\psi_{2k}(m)=\psi_k(m)+\phi_k(m).$$
\end{lemma}
To state the following Lemma, for all $t\in\Z$ we define
$$N_p(t)=\#\big\{(x,y): 0<x<y<\frac{p}{2},\(\frac{x}{p}\)=\(\frac{y}{p}\)=1,y^2-x^2\equiv t\pmod p\big\}.$$
We have the following result.
\begin{lemma}\label{lemma Np t}
Let $p\equiv1\pmod8$ be a prime of the form $a^2+4b^2$ with $a\equiv-1\pmod p$ and $b>0$. For all $t\in\Z$ with $p\nmid t$, we have
$$N_p(t)+N_p(-t)=\begin{cases}(p-3-2a+2\phi_2(t))/16&\mbox{if}\ (\frac{t}{p})=-1,\\
(p-7+2a+2\phi_2(t)-4\chi_4(t))/16&\mbox{if}\ (\frac{t}{p})=1.\end{cases}$$
\end{lemma}
\begin{proof}
It is clear that
$$N_p(-t)=\#\big\{(x,y):0<y<x<\frac{p}{2},\(\frac{x}{p}\)=\(\frac{y}{p}\)=1,y^2-x^2\equiv t\pmod p\big\}$$
and hence we see that $N_p(t)+N_p(-t)$ is equal to $n_p(t)/4$, where
$$n_p(t)=\#\big\{(x,y): 0<x,y<p,\(\frac{x}{p}\)=\(\frac{y}{p}\)=1,y^2-x^2\equiv t\pmod p\big\}.$$
It follows from definition that
$$n_p(t)=\frac{1}{4}\sum_{x=1}^{p-1}\sum_{y=1}^{p-1}
\(1+\(\frac{x}{p}\)\)\(1+\(\frac{y}{p}\)\)\(1-\(\frac{y^2-x^2-t}{p}\)^2\).$$
To obtain $n_p(t)$, we consider the following sums respectively. We first consider
$$S_1=\sum_{x=1}^{p-1}\sum_{y=1}^{p-1}\(1-\(\frac{y^2-x^2-t}{p}\)^2\).$$
Clearly
$$S_1=\#\big\{(x,y): 0<x,y<p, y^2-x^2\equiv t\pmod p\big\}=p-3-2\(\frac{t}{p}\).$$
We next consider the sums
$$S_2=\sum_{x=1}^{p-1}\sum_{y=1}^{p-1}\(\frac{x}{p}\)\(1-\(\frac{y^2-x^2-t}{p}\)^2\),$$
and
$$S_3=\sum_{x=1}^{p-1}\sum_{y=1}^{p-1}\(\frac{y}{p}\)\(1-\(\frac{y^2-x^2-t}{p}\)^2\).$$
Suppose first that $t$ is a quadratic non-residue modulo $p$. Then we have
\begin{align*}
S_2&=\sum_{x=1}^{p-1}\(\frac{x}{p}\)\sum_{y=1}^{p-1}\(1-\(\frac{y^2-(x^2+t)}{p}\)^2\)
=\sum_{x=1}^{p-1}\(\frac{x}{p}\)\(1+\(\frac{x^2+t}{p}\)\)
\\&=\phi_2(t).
\end{align*}
Suppose now that $-t\equiv t_0^2\pmod p$ for some $t_0\in\Z$. Then we have
\begin{align*}
S_2&=\sum_{x\not\equiv\pm t_0\pmod p}\(\frac{x}{p}\)\sum_{y=1}^{p-1}\(1-\(\frac{y^2-(x^2+t)}{p}\)^2\)
\\&=\sum_{x\not\equiv\pm t_0\pmod p}\(\frac{x}{p}\)\(1+\(\frac{x^2+t}{p}\)\)=\phi_2(t)-2\chi_4(t).
\end{align*}
The last equality follows from
$$\(\frac{t_0}{p}\)=\chi_4(-t)=\chi_4(t).$$
With the essentially same method, we can also get the explicit value of $S_3$ and obtain
$$S_2=S_3=\begin{cases}\phi_2(t)&\mbox{if}\ (\frac{t}{p})=-1,
\\\phi_2(t)-2\chi_4(t)&\mbox{if}\ (\frac{t}{p})=1.\end{cases}$$
Now we consider the sum
$$S_4=\sum_{x=0}^{p-1}\sum_{y=0}^{p-1}\(\frac{xy}{p}\)\(1-\(\frac{y^2-x^2-t}{p}\)^2\).$$
Let
$$\left(\begin{matrix}x\\y \end{matrix}\right)=\left(\begin{matrix}-1/2 & 1/2\\ 1/2 & 1/2\end{matrix}\right)\left(\begin{matrix}u\\v\end{matrix}\right).$$
Clearly this is a $\F_p$-linear bijection from $\F_p\times\F_p$ onto $\F_p\times\F_p$. Hence
\begin{align*}
S_4&=\sum_{u=0}^{p-1}\sum_{v=0}^{p-1}\(\frac{u^2-v^2}{p}\)\(1-\(\frac{uv-t}{p}\)^2\)
\\&=\sum_{u=1}^{p-1}\(\frac{u^2-t^2/u^2}{p}\)=\sum_{u=1}^{p-1}\(\frac{u^4-t^2}{p}\)=\psi_{4}(-t^2).
\end{align*}
By Lemma \ref{lemma phi 2 m} we have
\begin{align*}
\psi_4(-t^2)&=\psi_4(t^2)=\psi_2(t^2)+\phi_2(t^2)=-2+2\(\frac{t}{p}\)a.
\end{align*}
The last equality follows from $\psi_2(t^2)=-2$ (cf. \cite[p. 63]{IR}). Clearly we have
$$n_p(t)=\frac{1}{4}(S_1+S_2+S_3+S_4).$$
By the explicit values of $S_i$ for $i=1,2,3,4$, we get the desired result.
\end{proof}
Recall that $0<a_1<a_2<\cdots<a_{(p-1)/4}<p/2$ are all quadratic residues modulo $p$ in the interval $(0,p/2)$. Now we consider the following product
\begin{equation}\label{equation product ai}
W_p:=\prod_{1\le i<j\le\frac{p-1}{4}}(a_j^2-a_i^2).
\end{equation}
\begin{lemma}\label{lemma congruence of product ai}
Let $p\equiv1\pmod8$ be a prime of the form $a^2+4b^2$ with $a\equiv-1\pmod4$ and $b>0$.
Then we have the following results.

{\rm (i)} If $2$ is a biquadratic residue modulo $p$, then we have
$$W_p\equiv\begin{cases}
(-1)^{\lambda_p+\frac{p-3-2a+8b}{32}}\cdot A_p\pmod p
&\mbox{if}\ p\equiv 9\pmod{16},
\\(-1)^{\lambda_p+\frac{p+29-2a+8b}{32}}\cdot A_p\pmod p&\mbox{if}\ p\equiv 1\pmod{16}.\end{cases}$$

{\rm (ii)} If $2$ is not a biquadratic residue modulo $p$, then we have
$$W_p\equiv\begin{cases}
(-1)^{\lambda_p+\frac{p-3-2a-8b}{32}}\cdot2bB_p/a\pmod p
&\mbox{if}\ p\equiv 9\pmod{16},
\\(-1)^{\lambda_p+\frac{p+29-2a-8b}{32}}\cdot 2bB_p/a\pmod p
&\mbox{if}\ p\equiv 1\pmod{16}.\end{cases}$$
\end{lemma}
\begin{proof}
We first set
$$d_p:=\#\big\{(x,y): 0<x<y<\frac{p}{2},\(\frac{x}{p}\)=\(\frac{y}{p}\)=1,\{y^2-x^2\}_p>\frac{p}{2}\big\}.$$
It follows from definition that
\begin{align*}
W_p\equiv \prod_{0<t<p}t^{N_p(t)}
\equiv(-1)^{d_p}\prod_{0<t<p/2}t^{N_p(t)+N_p(-t)}\pmod p.
\end{align*}
The last congruence follows from
$$\sum_{0<t<p/2}N_p(-t)=d_p.$$
Since
$$\{y^2-x^2\}_p>p/2\Leftrightarrow\sin\frac{2\pi(y^2-x^2)}{p}<0,$$
we have
$$(-1)^{\lambda_p}=(-1)^{d_p}.$$
Recall that $\Omega_{1}$ and $\Omega_{-1}$ are defined as in Section 1. We also define the following sets.
\begin{align*}
\Omega_R:&=\{0<x<p/2: x^{(p-1)/2}\equiv1\pmod p\},
\\ \Omega_{2b/a}:&=\{0<x<p/2: x^{(p-1)/4}\equiv2b/a\pmod p\},
\\ \Omega_{-2b/a}:&=\{0<x<p/2: x^{(p-1)/4}\equiv-2b/a\pmod p\}.
\end{align*}
By Lemma (\ref{lemma Np t}) the product
$W_p\pmod p$ is equal to
\begin{align*}
(-1)^{\lambda_p}
\prod_{t\in\Omega_1}t^{\frac{p-11+6a}{16}}\prod_{t\in\Omega_{-1}}t^{\frac{p-3-2a}{16}}
\prod_{t\in\Omega_{2b/a}}t^{\frac{p-3-2a+8b}{16}}\prod_{t\in\Omega_{-2b/a}}t^{\frac{p-3-2a-8b}{16}}\pmod p.
\end{align*}
Simplifying the above congruence, we obtain
\begin{align*}
W_p&\equiv(-1)^{\lambda_p}\(\frac{p-1}{2}!\)^{\frac{p-3-2a-8b}{16}}
\prod_{t\in\Omega_1}t^{\frac{a+b-1}{2}}\prod_{t\in\Omega_{-1}}t^{\frac{b}{2}}
\prod_{t\in\Omega_{2b/a}}t^b
\\&\equiv(-1)^{\lambda_p}\(\frac{p-1}{2}!\)^{\frac{p-3-2a-8b}{16}}
\prod_{t\in\Omega_R}t^{\frac{b}{2}}\prod_{t\in\Omega_1}t^{\frac{a-1}{2}}\prod_{t\in\Omega_{2b/a}}t^b\pmod p.
\end{align*}
We claim
\begin{equation}\label{equation Omiga 2b/a}
\prod_{t\in\Omega_{2b/a}}t^2\equiv (-1)^{\frac{p+7}{8}}\cdot2b/a\pmod p,
\end{equation}
and
\begin{equation}\label{equation Omiga 1}
\prod_{t\in\Omega_1}t^2\equiv (-1)^{\frac{p+7}{8}}\pmod p,
\end{equation}
and
\begin{equation}\label{equation Omiga -1}
\prod_{t\in\Omega_{-1}}t^2\equiv (-1)^{\frac{p-1}{8}}\pmod p.
\end{equation}
In fact, we first have
$$x^{\frac{p-1}{4}}-2b/a\equiv \prod_{t\in\Omega_{2b/a}}(x-t)(x+t)\pmod {p}.$$
This gives
$$(-1)^{\frac{p-1}{8}}\prod_{t\in\Omega_{2b/a}}t^2\equiv -2b/a\pmod p.$$
Hence (\ref{equation Omiga 2b/a}) holds. With the same method, (\ref{equation Omiga 1}) follows from
$$x^{\frac{p-1}{4}}-1\equiv \prod_{t\in\Omega_1}(x-t)(x+t)\pmod {p},$$
and (\ref{equation Omiga -1}) follows from
$$x^{\frac{p-1}{4}}+1\equiv \prod_{t\in\Omega_{-1}}(x-t)(x+t)\pmod {p}.$$
Since $p\equiv1\pmod8$, we have $2\mid b$. In view of the above, we obtain
$$W_p\equiv(-1)^{\lambda_p+\frac{p+7}{8}\cdot\frac{a+2b-3}{4}}
\(\frac{p-1}{2}!\)^{\frac{p-3-2a-8b}{16}}(2b/a)^{\frac{b}{2}}
\prod_{t\in\Omega_R}t^{\frac{b}{2}}\prod_{t\in\Omega_1}t\pmod p.$$
It is known that $2$ is a $4$-th power residue modulo $p$ if and only if the equation
$$x^2+64y^2=p$$ has integral solutions. With our notations, this is equivalent to $4\mid b$.
The readers may consult \cite{Cox} for details. Now we divide the remaining proof into two cases.
Recall that $A_p$ and $B_p$ are defined as in Section 1.

{\it Case} I. $2$ is a $4$-th power residue modulo $p$, i.e., $4\mid b$.

In this case, we first have $a\equiv (p-3)/2\pmod {16}$.
By computation we obtain
$$W_p\equiv\begin{cases}
(-1)^{\lambda_p+\frac{p-3-2a+8b}{32}}\cdot A_p\pmod p
&\mbox{if}\ p\equiv 9\pmod{16},
\\(-1)^{\lambda_p+\frac{p+29-2a+8b}{32}}\cdot A_p\pmod p&\mbox{if}\ p\equiv 1\pmod{16}.\end{cases}$$

{\it Case} II. $2$ is not a $4$-th power residue modulo $p$, i.e., $b\equiv2\pmod4$.

In this case, we first have $a\equiv(p+13)/2\pmod {16}$. Via computation we get
$$W_p\equiv\begin{cases}
(-1)^{\lambda_p+\frac{p-3-2a-8b}{32}}\cdot2bB_p/a\pmod p
&\mbox{if}\ p\equiv 9\pmod{16},
\\(-1)^{\lambda_p+\frac{p+29-2a-8b}{32}}\cdot 2bB_p/a\pmod p
&\mbox{if}\ p\equiv 1\pmod{16}.\end{cases}$$
In view of the above, we complete the proof.
\end{proof}
We now turn to some results involving primitive roots modulo $p$. For each $2$-adic integer $n$, the $2$-adic order of $n$ is denoted by $\nu_2(x)$. And we use the symbol $\zeta$ to denote the number
$e^{2\pi\sqrt{-1}/(p-1)}$.
\begin{lemma}\label{lemma product involving roots of unity}
For each positive integer $k$ with $k\mid p-1$. Let $P(x)=\prod_{1\le i<j\le n}(x^{kj}-x^{ki})$, and
let $n=(p-1)/k$. Then we have
$$P(\zeta)=\begin{cases}(-1)^{(n-2)/4}\cdot n^{n/2}&\mbox{if}\ \nu_2(n)=1,
\\(-1)^{(n-4)/4}\cdot\sqrt{-1}\cdot n^{n/2}&\mbox{if}\ \nu_2(n)>1,
\\(-1)^{(n-1)/4}\cdot n^{n/2}&\mbox{if}\ n\equiv1\pmod4,
\\(-1)^{(n+1)/4}\cdot\sqrt{-1}\cdot n^{n/2}&\mbox{if}\ n\equiv3\pmod4.\end{cases}$$
\end{lemma}
\begin{proof}
One can verify the following identities.
\begin{align*}
P(\zeta)^2&=\prod_{1\le i<j\le n}(\zeta^{kj}-\zeta^{ki})^2
=(-1)^{\frac{n(n-1)}{2}}\cdot\prod_{1\le i\ne j\le n}(\zeta^{kj}-\zeta^{ki})
\\&=(-1)^{\frac{n(n-1)}{2}}\prod_{1\le j\le n}\frac{x^n-1}{x-\zeta^{kj}}{\Big|}_{x=\zeta^{kj}}
=(-1)^{\frac{n(n-1)}{2}}\cdot n^n\cdot\prod_{1\le j\le n}\zeta^{-kj}
\\&=(-1)^{\frac{n^2+n+2}{2}}\cdot n^n.
\end{align*}
The last equality follows from the identity $\prod_{1\le j\le n}(x-\zeta^{kj})=x^n-1$.
From the above, it is enough to determine $\Arg(P(\zeta))$, where $\Arg(z)$ denotes the argument of a complex number $z$. Via computation, we have
\begin{align*}
P(\zeta)&=\prod_{1\le i<j\le n}\zeta^{\frac{k(i+j)}{2}}(\zeta^{\frac{k(j-i)}{2}}-\zeta^{\frac{-k(j-i)}{2}})
\\&=\prod_{1\le i<j\le n}\zeta^{\frac{k(i+j)}{2}}\cdot2\sqrt{-1}\cdot\sin\frac{\pi(j-i)}{n}.
\end{align*}
Since $\sin\frac{\pi(j-i)}{n}>0$ for all $1\le i<j\le n$, we have
\begin{align*}
\Arg(P(\zeta))&\equiv \frac{\pi}{2}\cdot\frac{n(n-1)}{2}+\frac{\pi}{n}\cdot\sum_{1\le i<j\le n}(i+j)
\\ &\equiv \frac{\pi}{2}\cdot\frac{n(n-1)}{2}+\frac{\pi}{2n}\cdot\(\sum_{1\le i,j\le n}(i+j)-\sum_{1\le i\le n}2i\)\pmod{2\pi\Z}.
\end{align*}
By computation we obtain that
\begin{equation}\label{equation of argument}
\Arg(P(\zeta))\equiv \pi\cdot\frac{n(n+1)}{2}+\frac{\pi}{2}\cdot\frac{n^2-3n-2}{2}\pmod {2\pi\Z}.
\end{equation}
Then (\ref{equation of argument}) implies
$$P(\zeta)=(-1)^{\frac{n(n+1)}{2}}\cdot(\sqrt{-1})^{\frac{n^2-3n-2}{2}}\cdot n^{\frac{n}{2}}.$$
Now we can easily get the desired result.
\end{proof}
From Lemma \ref{lemma product involving roots of unity}, we immediately obtain the following result.
\begin{lemma}\label{lemma division by cyclotomic polynomial}
Let $G(x)$ be an integral polynomial defined by
$$G(x)=\begin{cases}(-1)^{(n-2)/4}\cdot n^{n/2}&\mbox{if}\ \nu_2(n)=1,
\\(-1)^{(n-4)/4}\cdot n^{n/2}\cdot x^{(p-1)/4}&\mbox{if}\ \nu_2(n)>1.\end{cases}$$
Then $\Phi_{p-1}(x)\mid P(x)-G(x)$.
\end{lemma}
\begin{proof}
It is sufficient to show $P(\zeta)=G(\zeta)$. Hence our desired result follows from Lemma \ref{lemma product involving roots of unity}.
\end{proof}
Now we are in the position to prove our main result.

\noindent{\bf Proof of Theorem \ref{theorem A}.} It follows from definition that
$$\sgn(\tau_p(g))=\prod_{1\le i<j\le \frac{p-1}{4}}\frac{\{a_j^2\}_p-\{a_i^2\}_p}{\{g^{4j}\}_p-\{g^{4i}\}_p}
\equiv\frac{W_p}{\prod_{1\le i<j\le \frac{p-1}{4}}(g^{4j}-g^{4i})}\pmod p.$$
By Lemma \ref{lemma Np t} we get the explicit value of $W_p\pmod p.$

Now we turn to the
denominator. By \cite[Proposition 10.3, p.61]{N} we see that
$p$ totally splits in $\Q(\zeta)$. Hence by the Kummer Theorem (cf. \cite[Proposition 8.3, p.47]{N}) we know that
$\Phi_{p-1}(x)\pmod {p}$ also totally splits in $\Z/p\Z[x]$.
And the set of primitive $(p-1)$-th roots of unity of $\Q(\zeta)$ maps bijectively onto the set of primitive $(p-1)$-th roots of
unity of $(\Z/p\Z)^{\times}$. We therefore have
$$\Phi_{p-1}(x)\equiv \prod_{g\in\mathcal{P}}(x-g)\pmod {p},$$
where
$$\mathcal{P}=\{0<g<p: \text{$g$ is a primitive root modulo $p$}\}.$$
By Lemma \ref{lemma division by cyclotomic polynomial} we have
$$\prod_{1\le i<j\le\frac{p-1}{4}}(g^{4j}-g^{4i})\equiv G(g)\pmod p.$$
Via computation, we obtain
$$\prod_{1\le i<j\le\frac{p-1}{4}}(g^{4j}-g^{4i})\equiv
\begin{cases}(-1)^{\lfloor\frac{p}{16}\rfloor+1}\cdot\chi_4(2)\pmod p&\mbox{if}\ p\equiv9\pmod{16},
\\(-1)^{\lfloor\frac{p}{16}\rfloor}\cdot\chi_4(2)\cdot g^{\frac{p-1}{4}}\pmod p&\mbox{if}\ p\equiv1\pmod{16},\end{cases}$$
where $\lfloor\cdot\rfloor$ denotes the floor function. By Lemma \ref{lemma Np t} and the above, we can get the explicit value of $\sgn(\tau_p(g))$.

When $p\equiv9\pmod {16}$,
it is clear that $\sgn(\tau_p(g))$ is independent on the choice of $g$.

In the case $p\equiv1\pmod{16}$, we
claim that for each $g\in\mathcal{P}$ we have
$$\sgn(\tau_p(g))\cdot\sgn(\tau_p(g^{-1}))=-1,$$\
where $g^{-1}\in\mathcal{P}$ satisfies $g\cdot g^{-1}\equiv1\pmod p$.
In fact, we have
\begin{align*}
\sgn(\tau_{p}(g))\cdot\sgn(\tau_{p}(g^{-1}))&\equiv
\prod_{1\le i<j\le \frac{p-1}{4}}\frac{g^{4j}-g^{4i}}{g^{-4j}-g^{-4i}}
\equiv\frac{G(g)}{G(g^{-1})}\\&\equiv g^{\frac{p-1}{2}}\equiv-1\pmod p.
\end{align*}
From this we obtain
$$\#\{g\in\mathcal{P}:\sgn(\tau_p(g))=1\}=\#\{g\in\mathcal{P}:\sgn(\tau_p(g))=-1\}.$$
In view of the above, we complete the proof.\qed
\maketitle
\section{Proof of Theorem \ref{theorem B}}
\setcounter{lemma}{0}
\setcounter{theorem}{0}
\setcounter{corollary}{0}
\setcounter{remark}{0}
\setcounter{equation}{0}
\setcounter{conjecture}{0}
We begin with the following Lemma.
\begin{lemma}\label{lemma A m}
For each $1\le m\le p-1$, we let
$$A_m:=\big\{1\le x\le p-1: \(\frac{x}{p}\)=\(\frac{x+m}{p}\)=1\big\}.$$
Then we have
$$\#A_m=\(p-3-2\(\frac{m}{p}\)\)/4.$$
\end{lemma}
\begin{proof}
It follows from definition that
\begin{align*}
\#A_m&=\frac{1}{4}\sum\limits_{\substack{0<x<p\\x\not\equiv -m\pmod p}}
\(1+\(\frac{x}{p}\)\)\(1+\(\frac{x+m}{p}\)\)
\\&=\frac{1}{4}\(p-2-2\(\frac{m}{p}\)+\sum_{0<x<p}\(\frac{x^2+mx}{p}\)\)
\\&=\frac{1}{4}\(p-3-2\(\frac{m}{p}\)\).
\end{align*}
The last equality follows from (cf. \cite[Exercise 8, p. 63]{IR})
$$\sum_{0<x<p}\(\frac{x^2+mx}{p}\)=-1.$$
\end{proof}
\begin{lemma}\label{lemma sum involving rational 4-th residues}
Let $A_m$ be defined as in Lemma \ref{lemma A m}. Let $p\equiv1\pmod8$ be a prime of the form $a^2+4b^2$ with $a\equiv-1\pmod4$ and $b>0$. Then we have
$$\sum_{x\in A_m}\(\frac{x^2+mx}{p}\)_4=\frac{1}{2}\(-1+a\(\frac{m}{p}\)\).$$
\end{lemma}
\begin{proof}
It is easy to see that
\begin{align*}
\sum_{x\in A_m}\(\frac{x^2+mx}{p}\)_4&=\frac{1}{4}\sum_{x=0}^{p-1}\sum_{y=0}^{p-1}
\(\frac{x^2y^2}{p}\)_4\(1-\(\frac{y^2-x^2-m}{p}\)^2\)
\\&=\frac{1}{4}\sum_{x=0}^{p-1}\sum_{y=0}^{p-1}
\(\frac{xy}{p}\)\(1-\(\frac{y^2-x^2-m}{p}\)^2\)=\frac{1}{4}S_4,
\end{align*}
where $S_4$ is defined as in the proof of Lemma \ref{lemma Np t}. Then our desired result follows from
the proof of Lemma \ref{lemma Np t}.
\end{proof}

We also need the following result.
\begin{lemma}\label{lemma 4-th residues in A m}
Let $A_m$ be defined as in Lemma \ref{lemma A m}, and let $p\equiv1\pmod8$ be a prime of the form
$a^2+4b^2$ with $a\equiv-1\pmod4$ and $b>0$. Then we have the following results.

{\rm (i)} If $(\frac{m}{p})=-1$, then
$$\#\big\{x\in A_m:\ \(\frac{x}{p}\)_4=1\big\}=\begin{cases}(p-1+4b)/8&\mbox{if}\ m^{\frac{p-1}{4}}\equiv 2b/a\pmod p,\\(p-1-4b)/8&\mbox{if}\ m^{\frac{p-1}{4}}\equiv -2b/a\pmod p.\end{cases}$$

{\rm (ii)} If $(\frac{m}{p})=1$, then
$$\#\big\{x\in A_m:\ \(\frac{x}{p}\)_4=1\big\}=\begin{cases}(p-7+2a)/8&\mbox{if}\ \chi_4(m)=1,\\(p-3-2a)/8&\mbox{if}\ \chi_4(m)=-1.\end{cases}$$
\end{lemma}
\begin{proof}
We divide our proof into the following two cases.

{\rm (i)} $(\frac{m}{p})=-1$. It is easy to see that
\begin{align*}
\#\big\{x\in A_m: \(\frac{x}{p}\)_4=1\big\}&=\frac{1}{8}\sum_{x=1}^{p-1}
\(\(\frac{x^2}{p}\)_4+1\)\(\(\frac{x^2+m}{p}\)+1\)
\\&=\frac{1}{8}(p-1+\phi_2(m)).
\end{align*}
Then our desired result follows from Lemma \ref{lemma phi 2 m}.

{\rm (ii)} $(\frac{m}{p})=1$. Here we set $m_0^2\equiv -m\pmod p$. It is easy to see that
\begin{align*}
\#\big\{x\in A_m: \(\frac{x}{p}\)_4=1\big\}&=
\sum\limits_{\substack{1\le x\le p-1\\ x\not\equiv\pm m_0\pmod p}}
\frac{1}{8}\(1+\(\frac{x^2}{p}\)_4\)\(1+\(\frac{x^2+m}{p}\)\)
\\&=\frac{1}{8}\(p-5-2\(\frac{m}{p}\)_4+\phi_2(m)\).
\end{align*}
Then Lemma \ref{lemma phi 2 m} implies our desired result.
\end{proof}
We are now in the position to prove our main result.

{\bf Proof of Theorem \ref{theorem B}}. For all $1\le m\le p-1$, we first let
$$L_m:=\big\{1\le x\le p-1-m: \(\frac{x}{p}\)=\(\frac{x+m}{p}\)=1\big\}.$$
We also let
\begin{align*}
r_m^{++}&=\#\big\{x\in L_m: \(\frac{x}{p}\)_4=\(\frac{x+m}{p}\)_4=1\big\},
\\r_m^{--}&=\#\big\{x\in L_m: \(\frac{x}{p}\)_4=\(\frac{x+m}{p}\)_4=-1\big\},
\\r_m^{+-}&=\#\big\{x\in L_m: \(\frac{x}{p}\)_4=1,\(\frac{x+m}{p}\)_4=-1\big\},
\\r_m^{-+}&=\#\big\{x\in L_m: \(\frac{x}{p}\)_4=-1,\(\frac{x+m}{p}\)_4=1\big\}.
\end{align*}
Using the map $x\mapsto p-m-x$, it is easy to see that
$$r_m^{+-}=r_m^{-+}.$$
Clearly we have the following equalities:
\begin{align*}
r_m^{++}+r_m^{--}+r_m^{+-}+r_m^{-+}&=\#L_m,
\\r_m^{++}+r_m^{--}-r_m^{+-}-r_m^{-+}&=\sum_{x\in L_m}\(\frac{x^2+mx}{p}\)_4,
\\r_m^{++}+r_m^{+-}&=\#\big\{x\in L_m: \(\frac{x}{p}\)_4=1\big\},
\\r_m^{--}+r_m^{-+}&=\#L_m-\#\big\{x\in L_m: \(\frac{x}{p}\)_4=1\big\}.
\end{align*}
From the above, by computation we obtain
\begin{equation}\label{equation r m ++}
r_m^{++}=\#\big\{x\in L_m: \(\frac{x}{p}\)_4=1\big\}-\frac{1}{4}\#L_m+\frac{1}{4}\sum_{x\in L_m}\(\frac{x^2+mx}{p}\)_4.
\end{equation}
Replacing $m$ by $p-m$, we get
\begin{equation*}
r_{p-m}^{++}=\#\big\{x\in L_{p-m}: \(\frac{x}{p}\)_4=1\big\}-\frac{1}{4}\#L_{p-m}+\frac{1}{4}\sum_{x\in L_{p-m}}\(\frac{x^2-mx}{p}\)_4.
\end{equation*}
Using the map $x\mapsto p-x$, we have
$$\#L_{p-m}=\#\big\{p+1-m\le x\le p-1: \(\frac{x}{p}\)=\(\frac{x+m}{p}\)=1\big\}.$$
We therefore get
$$\#L_m+\#L_{p-m}=\#A_m=\(p-3-2\(\frac{m}{p}\)\)/4.$$
The last equality follows from Lemma \ref{lemma A m}.
Using the same method, we also have
$$\sum_{x\in L_m}\(\frac{x^2+mx}{p}\)_4+\sum_{x\in L_{p-m}}\(\frac{x^2-mx}{p}\)_4=
\sum_{x\in A_m}\(\frac{x^2+mx}{p}\)_4.$$
From Lemma \ref{lemma sum involving rational 4-th residues} we have
$$\sum_{x\in L_m}\(\frac{x^2+mx}{p}\)_4+\sum_{x\in L_{p-m}}\(\frac{x^2-mx}{p}\)_4=
\frac{1}{2}\(-1+a\(\frac{m}{p}\)\).$$
As $-1$ is a $4$-th power residue modulo $p$, we can verify that
$$\#\big\{x\in A_m: \(\frac{x}{p}\)_4=1\big\}$$
is equal to
$$\#\big\{x\in L_m: \(\frac{x}{p}\)_4=1\}+\#\{x\in L_{p-m}: \(\frac{x}{p}\)_4=1\big\}.$$
By Lemma \ref{lemma 4-th residues in A m} we can easily get the explicit value of
$$\#\big\{x\in L_m: \(\frac{x}{p}\)_4=1\big\}+\#\big\{x\in L_{p-m}: \(\frac{x}{p}\)_4=1\big\}.$$
Let $b_i$ be defined as in (\ref{sequence F}) for $i=1,\cdots,(p-1)/4$. We clearly have
$$S_p:=\prod_{1\le i<j<\frac{p-1}{4}}(b_j-b_i)\equiv (-1)^{\sum_{0<m<p/2}r_{p-m}^{++}}\prod_{0<m<p/2}m^{r_m^{++}+r_{p-m}^{++}}\pmod p.$$
It is easy to verify that
\begin{align*}
\sum_{0<m<p/2}r_{p-m}^{++}&=\sum_{0<m<p/2}\#\big\{1\le x\le m-1: \(\frac{x}{p}\)_4=\(\frac{m-x}{p}\)_4=1\big\}
\\&=\varepsilon_p.
\end{align*}
In view of the above, via computation, we obtain
$$S_p\equiv(-1)^{\ve_p}\(\frac{p-1}{2}!\)^{\frac{p-3-2a-8b}{16}}
\prod_{m\in\Omega_R}m^{\frac{b}{2}}\prod_{m\in\Omega_1}m^{\frac{a-1}{2}}\prod_{m\in\Omega_{2b/a}}m^b\pmod p.$$
It follows from definition that
$$\sgn(\rho_p)\equiv \frac{S_p}{W_p}\pmod p,$$
where $W_p$ is defined as in (\ref{equation product ai}). Recall that in the proof of Lemma \ref{lemma congruence of product ai} we obtain
$$W_p\equiv(-1)^{\lambda_p}\(\frac{p-1}{2}!\)^{\frac{p-3-2a-8b}{16}}
\prod_{t\in\Omega_R}t^{\frac{b}{2}}\prod_{t\in\Omega_1}t^{\frac{a-1}{2}}\prod_{t\in\Omega_{2b/a}}t^b\pmod p.$$
Hence we finally obtain
$$\sgn(\rho_p)=(-1)^{\lambda_p+\ve_p}.$$
This completes the proof.\qed

\Ack\ We would like to thank the referees for their helpful comments.
This research was supported by the National Natural Science Foundation of
China (Grant No. 11971222).

\end{document}